\documentclass[11pt]{amsart}
\usepackage{amssymb, amstext, amscd, amsmath, amssymb}
\usepackage{mathtools, xypic, paralist, color, dsfont, rotating}
\usepackage{verbatim}
\usepackage{enumerate,enumitem}
\usepackage{float}
\usepackage{setspace}

\usepackage{tikz}

\floatstyle{boxed} 
\restylefloat{figure}

\numberwithin{equation}{section}
\usepackage[a4paper]{geometry}
\usepackage{quoting}
\quotingsetup{vskip=.1in}
\quotingsetup{leftmargin=.17in}
\quotingsetup{rightmargin=.17in}

\let\OLDthebibliography\thebibliography
\renewcommand\thebibliography[1]{
  \OLDthebibliography{#1}
  \setlength{\parskip}{0pt}
  \setlength{\itemsep}{2pt plus 0.5ex}
}

\makeatletter
\def\@cite#1#2{{\m@th\upshape\bfseries%
[{#1\if@tempswa{\m@th\upshape\mdseries, #2}\fi}]}}
\makeatother
\theoremstyle{plain}
\newtheorem{theorem}{Theorem}[section]
\newtheorem{corollary}[theorem]{Corollary}
\newtheorem{proposition}[theorem]{Proposition}
\newtheorem{lemma}[theorem]{Lemma}
\theoremstyle{definition}
\newtheorem{definition}[theorem]{Definition}

\newtheorem{question}[theorem]{Question}

\theoremstyle{remark}

\mathtoolsset{centercolon}
\newcommand{\bbC}{{\mathbb{C}}}

\newcommand{\bbF}{{\mathbb{F}}}

\newcommand{\bbT}{{\mathbb{T}}}
\newcommand{\bbZ}{{\mathbb{Z}}}

\newcommand{\A}{{\mathcal{A}}}
\newcommand{\B}{{\mathcal{B}}}

\renewcommand{\H}{{\mathcal{H}}}

\newcommand{\T}{{\mathcal{T}}}

\newcommand{\X}{{\mathcal{X}}}
\newcommand{\Y}{{\mathcal{Y}}}

\renewcommand{\phi}{\varphi}
\newcommand{\upchi}{{\raise.35ex\hbox{\ensuremath{\chi}}}}

\newcommand{\Aut}{\operatorname{Aut}}

\newcommand{\id}{{\operatorname{id}}}

\newcommand\diag{\mathop{\rm diag}}

\newcommand{\ca}{\mathrm{C}^*}

\begin{document}
\title[Isomorphism of tensor algebras]{The isomorphism problem for tensor algebras of multivariable dynamical systems}

\author[E.G. Katsoulis]{Elias~G.~Katsoulis}
\address {Department of Mathematics 
\\East Carolina University\\ Greenville, NC 27858\\USA}
\email{katsoulise@ecu.edu}

\author[C. Ramsey]{Christopher~Ramsey}
\address {Department of Mathematics and Statistics
\\MacEwan University \\ Edmonton, AB \\Canada}
\email{ramseyc5@macewan.ca}


\maketitle
\begin{abstract}
We resolve the isomorphism problem for  tensor  algebras of unital multivariable dynamical systems. Specifically we show that unitary equivalence after a conjugation for multivariable dynamical systems is a complete invariant for complete isometric isomorphisms between their tensor algebras. In particular, this settles a conjecture of Davidson and Kakariadis relating to work of Arveson from the sixties, and extends related work of Kakariadis and  Katsoulis.
\end{abstract}

\section{Introduction}

Semicrossed products and their variants have appeared in the theory of operator algebras since the beginning of the subject \cite{Arv1} and continue to be at the forefront of the theory as they lend insight for considerable abstraction \cite{DavFulKak2, DavKak, DavKat1, Hum20, KatRamMem, Ramsey, Wi}.  

A $\ca$-dynamical system $(\A, \alpha)$ consists of a unital $\ca$-algebra $\A$ and a unital $*$-endomorphism $\alpha : \A\rightarrow \A$. An isometric covariant representation $(\pi, V)$ of $(\A, \alpha)$
consists of a non-degenerate  $*$-representation $\pi$ of $\A$ on a Hilbert
space $\H$ and an isometry $V \in B(\H )$ so that $\pi(a)V=V\pi(\alpha(a))$, for all $ a \in \A$.  The semicrossed product
$\A\rtimes_{\alpha}\bbZ^{+}$ is the universal operator algebra associated with ``all" covariant representations of $(\A, \alpha)$, i.e., the universal algebra generated
by a copy of $\A$ and an isometry $v$ satisfying the covariance relations. In the case where $\alpha$ is an automorphism of $\A$, then $\A\rtimes_{\alpha}\bbZ^{+}$  is isomorphic to the subalgebra of the crossed product $\ca$-algebra $\A\rtimes_{\alpha}\bbZ$ generated by $\A$ and the ``universal" unitary $u$ implementing the covariance relations. 

One of the central problems in the study of semicrossed products is the classification problem, whose study spans more than 50 years. This problem asks if two semicrossed products are isomorphic as algebras exactly when the corresponding $\ca$-dynamical systems are outer conjugate, that is, unitarily equivalent after a conjugation. The classification problem first appeared in the works of Arveson \cite{Arv1} and Arveson and  Josephson \cite{ArvJ}. It was subsequently investigated by Peters \cite{Peters}, Hadwin and Hoover \cite{HadHoo}, Power \cite{Pow1} and  Davidson and Katsoulis \cite{DavKat1}, who finally settled the case where $\A$ is abelian. In the general case where $\A$ may not  be abelian, initial consideration for the isomorphism problem was given in \cite{DavKat1.5, MS00} and considerable progress was made by Davidson and  Kakariadis \cite{DavKak} who resolved the problem for \textit{isometric isomorphisms} and dynamical systems consisting of injective endomorphisms. Actually  the work of Davidson and Kakariadis went well beyond systems consisting of injective endomorphisms. In \cite[Theorem 1.1]{DavKak} these authors also worked the case of epimorphic systems and in \cite[Theorem 1.2]{DavKak} they offered 6 additional conditions, with each one of them guaranteeing  a positive resolution for the isomorphism problem. All these partial results offered enough  evidence for Davidson and Kakariadis to conjecture\footnote{See just below Theorem 1.1 in \cite{DavKak}.} that the isomorphism problem for isometric isomorphisms must  have a positive solution for arbitrary systems. This conjecture is now being verified here in Corollary~\ref{cor;main}, thus resolving the isomorphism problem for unital dynamical systems and their semicrossed products at the level of isometric isomorphisms. Our initial approach is different from that of  \cite{DavKak} and actually allows us to achieve more.

A \textit{multivariable $\ca$-dynamical system} is a pair $(\A, \alpha)$ consisting of a unital $\ca$-algebra $\A$ along with unital $*$-endomorphisms $\alpha = (\alpha_1, \dots, \alpha_n)$ of $\A$ into itself. A row isometric representation of $(\A, \alpha)$ consists of a non-degenerate  $*$-representation $\pi$ of  $\A$ on a Hilbert space $\H$ and a row isometry $V = (V_1, V_2, \dots , V_n)$ acting on $\H^{(n)}$ so that $\pi(a)V_i=V_i\pi(\alpha(a))$, for all $a\in \A$ and $i = 1, 2, \dots, n$. The tensor algebra $\T^{+}(\A, \alpha)$ is the universal algebra generated
by a copy of $\A$ and a row isometry $v$ satisfying the covariance relations. The tensor algebras form a tractable multivariable generalization of the semicrossed products; indeed in the case where $n=1$, a multivariable system consists of a single endomorphism $\alpha$ and the corresponding tensor algebra $\T^{+}(\A, \alpha)$ coincides with the semicrossed product $\A\rtimes_{\alpha}\bbZ^{+}$. One should be careful to note that there are semicrossed products of multivariable $\ca$-dynamical systems as well, but these are different from the tensor algebras being discussed here. (See \cite{ DavKat2, KakK, Ramsey} for more information.)

In the case where $\A$ is abelian, tensor algebras of multivariable systems were first studied in detail by Davidson and Katsoulis \cite{DavKat2}. These authors developed a satisfactory dilation theory and provided invariants of a topological nature for algebraic isomorphisms, which in certain cases turned out to be complete.  Specifically, in~\cite[Definition 3.22]{DavKat2} Davidson and Katsoulis introduced the concept of piecewise conjugacy for classical multivariable dynamical systems and in \cite[Theorem 3.22]{DavKat2} they established that if the tensor algebras are algebraically isomorphic then the two dynamical systems must be piecewise conjugate. However, the converse could only be established for tensor algebras with $n =2 $ or $3$ \cite[Proposition 3.22]{DavKat2} with a gap in the topological theory preventing a complete result.

The tensor algebras of perhaps non-abelian multivariable systems were studied in \cite{KakK, KakK12}. In general, the piecewise conjugacy of Davidson and Katsoulis does not generalize to this context and even when it does, e.g. automorphic multivariable systems, it may not form a complete invariant for isometric isomorphisms~\cite[Example 4.12]{KakK}. In \cite[Theorem 4.5(ii)]{KakK} it was established that another invariant, the unitary equivalence after a conjugation (see Definition \ref{maindefn} below) actually forms a complete invariant for isometric isomorphisms between tensor algebras of \textit{automorphic} multivariable systems. Under various assumptions Kakariadis and  Katsoulis were able to  show that the classification scheme \cite[Theorem 4.5(ii)]{KakK} could be extended  to more general dynamical systems  but a complete solution was not obtained in \cite{KakK}. (This was pointed out as an open problem in \cite[Question 1]{KakK} and \cite{Kat1}, just after Theorem 2.6.15). All these assumptions are now being  removed in this paper and in Theorem~\ref{T;main} we obtain a complete classification of  all tensor algebras of (unital) multivariable systems up to completely isometric isomorphism.

The semicrossed products of single variable systems and more  generally  the tensor  algebras of multivariable  systems are prototypical examples of tensor algebras of $\ca$-correspondences. These algebras were pioneered by Muhly and Solel \cite{MS1} and generalize many concrete classes of operator algebras, including graph algebras and more. The isomorphism problem generalizes in this setting and asks if two tensor algebras of $\ca$-correspondences are isomorphic exactly when the $\ca$-correspondences are unitarily equivalent. Muhly and Solel studied this problem in \cite{MS00} and resolved it affirmatively for aperiodic $\ca$-correspondences and complete isomorphisms. However many natural examples of $\ca$-correspondences, including those associated with tensor algebras of multivariable systems or graph algebras, may fail to be aperiodic. As it turns out, the unitary equivalence of the correspondences associated with tensor algebras of multivariable systems coincides with unitary equivalence after a conjugation for the multivariable systems themselves. Therefore our Theorem~\ref{T;main} resolves Muhly and Solel's classification problem for an important class of $\ca$-correspondences and lends support for an overall affirmative answer of this problem. We plan to pursue this in a subsequent work.

Finally a word about our notation. If $\rho : \X \rightarrow \Y$ is any map between linear spaces, then its $(m, n)$-th ampliation is the matricial map 
\[
\rho^{(m, n)}: M_{m,n}(\X)\longrightarrow M_{m,n}(\Y); {[x_{i j}]_{i=1}^m} _{j=1}^{n} \longmapsto {[\rho(x_{i,j})]_{i=1}^m} _{j=1}^{n}.
\]
In what  follows, in order to avoid the use of heavy notation we will be dropping the superscript $(m,n)$ from $\rho^{(m,n)}$ and the symbol $\rho$ will be used not only for the map $\rho$ itself but for all of its ampliations as well. It goes without saying that the order of an ampliation will be easily understood form the context. 

\section{The main result}

Suppose $\mathcal A$ is a unital C$^*$-algebra and $\alpha_i:\mathcal A\rightarrow\mathcal A, 1\leq i\leq n$, are unital $*$-endomorphisms. Recall that the tensor algebra $\T^+(\A, \alpha)$ of the C$^*$-dynamical system $(\mathcal A,\alpha)$ is generated by $v_1, \dots, v_n$, which form a row isometry $v = [v_1\ v_2\ \dots\ v_n]$ and a faithful copy of $\mathcal A$. The  row isometry $v$ encodes the dynamics of the dynamical system $(\mathcal A,\alpha)$ in the sense that $av_i=v_i\alpha_i(a)$, for all $a \in \A$ and $1\leq i \leq n$. By definition, the tensor algebra $\T^+(\A, \alpha)$ is universal over all representations encoding the dynamics of $(\mathcal A,\alpha)$ so that the generating isometries $v_1, v_2, \dots , v_n$ are mapped to a row isometry.

Due to its universality, the tensor algebra $\T^+(\A, \alpha) $ admits a gauge action 
$$\zeta:\bbT \longrightarrow \Aut (\T^+(\A, \alpha)); \ \ \bbT\ni \lambda \longmapsto \zeta_{\lambda} $$ 
so that $\zeta_{\lambda} (a) = a $, for all $a \in \A$, and $\zeta_{\lambda} (v_p)=\lambda^{|p|} v_p$, where $p=i_1i_2\dots i_k \in \bbF^+_{n}$ is an element of the free semigroup with $n$ generators, $|p| :=k$ denotes the length of $p$ and $v_p:=v_{i_1}v_{i_2}\dots v_{i_k}$. (We write $p=0$ is the empty word, with the understanding that $|0| =0$ and $v_{0}:=I$.) From this gauge action we deduce that every element $x \in \T^+(\A, \alpha) $ admits a formal Fourier series development $x \sim \sum_{k =0}^{\infty} E_k(x)$. Each $E_k: \T^+(\A, \alpha) \rightarrow \T^+(\A, \alpha)$ is a completely contractive, $\A$-module projection on the subspace of $\T^+(\A, \alpha) $ generated by elements of the form $v_pa_p$, with $p \in \bbF^+_{n}$, $|p|=k$ and $a_p \in \A$. Furthermore $E_0$ is a multiplicative expectation onto $\A \subseteq \T^+(\A, \alpha)$. Finally the formal series  $x \sim \sum_{k =0}^{\infty} E_k(x)$ is Cesaro-convergent to $x $, i.e., 
\[
x = \lim_{n\rightarrow \infty} \sum_{k=0}^n \left(1 - \frac{k}{n+1}\right)E_k(x).
\]
See \cite{DavKat2} for a comprehensive development of this theory.

The following is a key result in our investigation.

\begin{theorem}\label{thm:mobius}
If $b = [b_1 \ \dots \ b_n]$ is a strict row contraction in $\mathcal A$ such that  $ab_i = b_i\alpha_i(a)$ for all $a\in \mathcal A$ and $1\leq i\leq n$ then there is a completely isometric automorphism $\rho$ of $\T^+(\A, \alpha)$ such that $\rho(a) = a$ for $a\in \mathcal A$,
\[
\rho(v) = (I - bb^*)^{1/2}(I-vb^*)^{-1}(b-v)(I_n-b^*b)^{-1/2}
\]
and $\rho\circ\rho = \id$. Furthermore, $E_0(\rho(v_i)) = b_i, 1\leq i\leq n$.
\end{theorem}

\begin{proof}
By hypothesis $b$ is a strict contraction and so $\|vb^*\| < 1$ which implies that $I-vb^*$ and $D_{b*} := (I-bb^*)^{1/2}$ are indeed invertible in $\mathcal A$. Similarly, $D_{b} := (I_n - b^*b)^{1/2}$ is invertible in $M_n(\mathcal A)$. One calculates
\begin{align*}
(I-bv^*)^{-1}&(I-bb^*)(I-vb^*)^{-1} 
\\ &= (I-bv^*)^{-1}(I-bv^*vb^*)(I-vb^*)^{-1}
\\ &= (I-bv^*)^{-1}(I  - vb^* + (I -bv^*)vb^*)(I-vb^*)^{-1}
\\ &= (I-bv^*)^{-1} + vb^*(I-vb^*)^{-1}
\\ &= (I-bv^*)^{-1} + (I-vb^*)^{-1} - I
\end{align*}
which gives for $w(v) := (I - bb^*)^{1/2}(I-vb^*)^{-1}(b-v)(I-b^*b)^{-1/2}$ that
\begin{align*}
w(v)^*w(v)
& = D_b^{-1}(b^*-v^*)(I-bv^*)^{-1}(I-bb^*)(I-vb^*)^{-1}(b-v)D_b^{-1}
\\ & = D_b^{-1}v^*(vb^*- I)\Big((I-bv^*)^{-1} + (I-vb^*)^{-1} - I\Big)(bv^*- I)vD_b^{-1}
\\ & = D_b^{-1}v^*\Big((I-vb^*) + (I-bv^*) - (I-vb^*)(I-bv^*)\Big)vD_b^{-1}
\\ & = D_b^{-1}v^*(I - vb^*bv^*)vD_b^{-1}
\\ & = D_b^{-1}(I_n-b^*b)D_b^{-1} 
\\ & = I_n.
\end{align*}
Hence, $w(v)$ is a row isometry. Now by the conjugation relation in the hypothesis we have $bb^*$, $bv^*$ and $vb^*$ commute with $\mathcal A$ and $b^*b$ commutes with $\diag(\alpha_1(a),\dots, \alpha_n(a)), a\in \mathcal A$. Using this one gets for $a\in \mathcal A$ that
\begin{align*}
aw & = a(I-bb^*)^{-1}(I-vb^*)^{-1}(b-v)(I_n-b^*b)^{-1/2}
\\ &  = (I - bb^*)^{1/2}(I-vb^*)^{-1}a(b-v)(I_n-b^*b)^{-1/2}
\\ & = (I - bb^*)^{1/2}(I-vb^*)^{-1}(b-v)\diag(\alpha_1(a), \dots, \alpha_n(a))(I-b^*b)^{-1/2}
\\ & = (I - bb^*)^{1/2}(I-vb^*)^{-1}(b-v)(I-b^*b)^{-1/2}\diag(\alpha_1(a), \dots, \alpha_n(a))
\\ & = w\diag(\alpha_1(a), \dots, \alpha_n(a)).
\end{align*}
Thus, by the universal property there exists a unique completely contractive homomorphism $\rho$ of $\T^+(\A, \alpha)$ to itself such that
\[
\rho(a) = a, \forall a\in \mathcal A \quad \textrm{and} \quad [\rho(v_1) \dots \rho(v_n)] = w(v).
\]

Lastly, recall the classical Halmos functional calculus trick: since
\[
b(I_n-b^*b) = (I-bb^*)b
\]
then
\[
bD_b^{-1} = b(I_n-b^*b)^{-1/2} = (I-bb^*)^{-1/2}b = D_{b^*}^{-1}b.
\]
This allows us to compute
\begin{align*}
[\rho&\circ\rho(v_1),  \dots ,\rho\circ\rho(v_n)] = w(w(v))
\\ & = D_{b^*}\Big(I - D_{b^*}(I-vb^*)^{-1}(b-v)D_b^{-1}b^*\Big)^{-1}\Big( b - D_{b^*}(I-vb^*)^{-1}(b-v)D_b^{-1} \Big)D_b^{-1}
\\ &= D_{b^*}\Big(I - D_{b^*}(I-vb^*)^{-1}(b-v)b^*D_{b^*}^{-1}\Big)^{-1}\Big( b - D_{b^*}(I-vb^*)^{-1}(b-v)D_b^{-1} \Big)D_b^{-1}
\\ &= D_{b^*}^2\Big(I - (I-vb^*)^{-1}(b-v)b^*\Big)^{-1}D_{b^*}^{-1}\Big( b - D_{b^*}(I-vb^*)^{-1}(b-v)D_b^{-1} \Big)D_b^{-1}
\\ &= D_{b^*}^2\Big(I - (I-vb^*)^{-1}(b-v)b^*\Big)^{-1}\Big( D_{b^*}^{-1}b - (I-vb^*)^{-1}(b-v)D_b^{-1} \Big)D_b^{-1}
\\ &= D_{b^*}^2\Big(I - (I-vb^*)^{-1}(b-v)b^*\Big)^{-1}\Big( b - (I-vb^*)^{-1}(b-v)\Big)D_b^{-2}
\\ &= D_{b^*}^2\Big((I-vb^*) - (b-v)b^*\Big)^{-1}(I-vb^*)\Big( b - (I-vb^*)^{-1}(b-v)\Big)D_b^{-2}
\\ &= D_{b^*}^2(I-bb^*)\Big((I-vb^*)b - (b-v)\Big)D_b^{-2}
\\ &= (v-vb^*b)D_b^{-2}
\\ &= v.
\end{align*}
Therefore, by the universal property $\rho\circ\rho = \id$ and $\rho$ is a completely isometric automorphism of $\T^+(\A, \alpha)$.

To find the first Fourier coefficients of the $\rho(v_i)$ one needs to recall that $E_0:\T^+(\A, \alpha) \rightarrow \mathcal A$ is a completely contractive homomorphism given by sending $\mathcal A$ to itself and sending $v_i$ to 0, $1\leq i\leq n$.
Thus,
\begin{align*}
E_0([\rho(v_1) , \dots , \rho(v_n)]) & = E_0(w(v))
\\ & = E_0\big((I - bb^*)^{1/2}(I-vb^*)^{-1}(b-v)(I_n-b^*b)^{-1/2}\big)
\\ & = (I - bb^*)^{1/2}((I-0b^*))^{-1}(b-0)(I_n-b^*b)^{-1/2}
\\ & = (I - bb^*)^{1/2}b(I_n-b^*b)^{-1/2}
\\ & = b
\end{align*}
with the last equality arising from a familiar functional calculus argument.
\end{proof}

In the proof of our next result we make use of the orbit representations for a tensor algebra $\T^+(\A, \alpha)$. Let $\sigma: \A \rightarrow B(\H)$ be a $*$-representation. We define a representation 
\[
\pi_{\sigma}:\A \longrightarrow B(\H \otimes \ell^2(\bbF^+_n))
\]
by 
\[
\pi_{\sigma}(a)(x\otimes e_p )= \pi(\alpha_p(a))x\otimes e_p, \quad a \in \A,
\]
where $\{e_p\}_{p \in \bbF^+_n}$ is the canonical orthonormal basis of the Fock space $\ell^2(\bbF^+_n)$. Let $L_1, L_2, \dots , L_n$ be the Cuntz-Toeplitz isometries on $\ell^2(\bbF^+_n)$. Then the representation $\pi_{\sigma}$ together with the isometries $I\otimes L_1, I\otimes L_2, \dots, I\otimes L_n$ form a covariant representation of $(\A, \alpha)$ and determine the orbit representation of $\T^+(\A, \alpha)$ associated with $\rho$, which we also denote as $\pi_{\sigma}$.

\begin{proposition} \label{prop;strict}
Let $\psi : \T^+(\A, \alpha) \rightarrow \T^+(\A, \beta)$ be a completely isometric isomorphism and assume that $\psi|_{\mathcal A} = \id$. If $v= [v_1, v_2, \dots , v_{n_a}]$ is the generating row isometry for $ \T^+(\A, \alpha) $, then $\|E_0(\psi(v))\|<1$. 
\end{proposition}

\begin{proof}
Let $\sigma: \T^+(\A, \beta) \longrightarrow B(\mathcal H)$ be a completely isometric representation of $T^+(\A, \beta)$. Then $\sigma\circ \psi$ is a completely isometric representation of $T^+(\A, \alpha)$, which coincides with $\sigma$ on $\A$. This allows us to view both tensor algebras as being embedded in $B(\H)$ via maps that coincide on $\A$.

Let $w = [ w_1, w_2, \dots , w_{n_b}]$ be the generating row isometry\footnote{Note that we are not assuming $n_a=n_b$} for  $\T^+(\A, \beta)$.  
By the Fourier analysis discussed earlier
\begin{equation} \label{eq;specFourier}
\begin{split}
\psi(v)&= \lim_{n\rightarrow \infty} \sum_{k=0}^n \left(1 - \frac{k}{n+1}\right)E_k(\psi(v))\\
&= \lim_{n\rightarrow \infty} \sum_{k=0}^n \sum_{|p| = k }\left(1 - \frac{k}{n+1}\right)w_pb_p 
\end{split}
\end{equation}
with each $b_p=[b_{p,1},b_{p,2},\dots ,b_{p,n_{a}}]$ being a row contraction.

Assume by contradiction that $\|b_0\| = \|E_0(\psi(v))\| = 1$. Then there exists a sequence of unit vectors $\xi_j \in \mathcal H^{(n_b)}$ such that $\lim_{j\rightarrow \infty}\|b_0\xi_j\| =1$. 
Using these contractive Cesaro sums in conjunction with the orbit representation $\pi:= \pi_{\sigma_{\mid \A}}$ applied to the vectors $\vec{\xi}_j=\xi_j\otimes e_0$ we get 
\begin{align*}
1  & \geq
  \lim_{j\rightarrow \infty} \left\| \pi\left(  \sum_{k=0}^n \sum_{|p| = k }\left(1 - \frac{k}{n+1}\right)w_pb_p  \right) \vec{\xi}_j \right\|^2
\\ & = \lim_{j \rightarrow \infty}  \sum_{k=0}^n \sum_{|p| = k } \left(1 - \frac{k}{n+1}\right)^2 \|b_p\xi_j\|^2   
\\ & \geq \lim_{j\rightarrow \infty} \|b_0\xi_j\|^2= 1.
\end{align*}
This implies that $\lim_{j\rightarrow \infty} \|b_p\xi_j\| = 0$, for all non-empty words $p \in \bbF^+_{n_a}$.

Pick an $n\geq0$ such that 
\[
\left\| \sum_{k=0}^n \left(1 - \frac{k}{n+1}\right)E_k(\psi(v))  - \psi(v) \right\| \leq \frac{1}{2}.
\]
Then for every $j\geq 1$ we have
\begin{align*}
\left\|  E_0\circ\psi^{-1}\left( \sum_{k=0}^n \left(1 - \frac{k}{n+1}\right)E_k(\psi(v))\right)\xi_j\right\| 
\\ & \hskip -130 pt \leq \left\|  E_0\circ\psi^{-1}\left(\sum_{k=0}^n \left(1 - \frac{k}{n+1}\right)E_k(\psi(v))- \psi(v)\right)\xi_j\right\| + \|E_0\circ\psi^{-1}\circ\psi(v)\xi_j\| 
\\ & \hskip -130 pt \leq \left\| \sum_{k=0}^n \left(1 - \frac{k}{n+1}\right)E_k(\psi(v))- \psi(v)\right\|\|\xi_j\| + \|E_0(v)\xi_j\|
\\ & \hskip -130 pt \leq \frac{1}{2} + 0 = \frac{1}{2}.
\end{align*}

\vspace{.03in}
\noindent Using the fact that $\lim_{j\rightarrow \infty} b_p\xi_j = 0$, for all non-empty words $p \in \bbF^+_{n_a}$, we obtain that

\begin{align*}
\lim_{j\rightarrow \infty} \left\| E_0\circ\psi^{-1}\left( \sum_{k=0}^n \left(1 - \frac{k}{n+1}\right)E_k(\psi(v))\right)\xi_j \right\|
\\ & \hskip -80 pt = \lim_{j\rightarrow \infty} \left\|E_0\left( \sum_{k=0}^n \sum_{|p| = k }\left(1 - \frac{k}{n+1}\right)\psi^{-1}(w_p) b_p   \right)\xi_j\right\|
\\ & \hskip -80 pt = \lim_{j\rightarrow \infty} \left\| \sum_{k=0}^n \sum_{|p| = k }\left(1 - \frac{k}{n+1}\right)E_0\big(\psi^{-1}(w_p)\big)  b_p  \xi_j\right\|
\\ & \hskip -80 pt =  \lim_{j\rightarrow \infty}  \left\| E_0 \big(\psi^{-1}(w_0)\big)  b_0\xi_j \right\|
\\ & \hskip -80 pt = 1,
\end{align*}
which is a contradiction. Therefore, $b_0 = E_0(\psi(v))$ must be a strict contraction.
\end{proof}

Continuing with the assumptions and notation of the previous proposition and its proof, the matrix $[b_{i j }] \in M_{n_b, n_a}(\A)$, with entries $b_{i j}$ as appearing in (\ref{eq;specFourier}) with $|p|=1$, is called the \textit{matrix associated with $\psi$}. It has the property of \textit{intertwining the endomorphisms} $\alpha_1, \alpha_2, \dots , \alpha_{n_a}$ and $\beta_1, \beta_2, \dots, \beta_{n_b}$ in the sense,
\begin{equation} \label{eq;bij}
\beta_i(a)b_{ij}=b_{ij}\alpha_j(a)
\end{equation}
for all $a\in \A$, $1\leq i \leq n_b$ and $1\leq j \leq n_a$. A matrix $[a_{i j }] \in M_{n_a, n_b}(\A)$ with similar properties is associated with $\psi^{-1}$. Let us observe this more closely. 

Indeed, assume that $\psi(v)$ admits a Fourier development
\begin{equation}  \label{eq;psi}
\begin{split}
\psi(v)&= \lim_{n\rightarrow \infty} \sum_{k=0}^n \sum_{|p| = k }\left(1 - \frac{k}{n+1}\right)w_pb_p \\
&= b_0+w[b_{ij}]+\lim_{n\rightarrow \infty} \sum_{k=2}^n \sum_{|p| = k }\left(1 - \frac{k}{n+1}\right)w_pb_p 
\end{split}
\end{equation}
as in (\ref{eq;specFourier}). For any $a \in \A$ and $1\leq j\leq n_a$, we have 
$$a\psi(v_j) = \psi(av_j)=\psi(v_j\alpha_j(a)) = \psi(v_j) \alpha_j(a)$$ 
and so 
\begin{equation} \label{eq;assoc}
\begin{split}
a  \Big( b_{0j}+\sum_{i=1}^{n_b}w_ib_{ij}+ \lim_{n\rightarrow \infty} \sum_{k=2}^n & \sum_{|p| = k }\left(1 - \frac{k}{n+1}\right)w_pb_{p j} \Big) =\\ &=\Big(b_{0j}+\sum_{i=1}^{n_b}w_ib_{ij}+\sum_{k=2}^n \sum_{|p| = k }\left(1 - \frac{k}{n+1}\right)w_pb_{pj} \Big) \alpha_j(a)
\end{split}
\end{equation}
Apply $E_0$ to (\ref{eq;assoc}) to obtain 
\begin{equation} \label{eq;mainrel}
ab_{0j}=b_{0j} \alpha_j(a),  
\end{equation}
for all $a \in \A$ and $1\leq j \leq n_a$. The relations (\ref{eq;bij}) are obtained by applying $E_1$ to (\ref{eq;assoc}).

\begin{lemma} \label{lem;inv}
Let $\psi : \T^+(\A, \alpha) \rightarrow \T^+(\A, \beta)$ be a completely isometric isomorphism and assume that $\psi|_{\mathcal A} = \id$. Let $v$ and $w$ be the generating row isometries for $ \T^+(\A, \alpha) $ and $ \T^+(\A, \beta) $ respectively. Then $E_0(\psi(v))=0$ if and only if $E_0(\psi^{-1}(w))=0$. In such a case, the matrices associated with $\psi$ and $\psi^{-1}$ are inverses of each other.
\end{lemma}

\begin{proof}
Assume that $\psi(v)$ admits a Fourier development as in (\ref{eq;psi})
and similarly
\begin{equation*} 
\psi^{-1}(w)=  a_0+v[a_{ij}]+\lim_{n\rightarrow \infty} \sum_{k=2}^n \sum_{|q| = k }\left(1 - \frac{k}{n+1}\right)v_qa_q 
\end{equation*}
Assume that $E_0(\psi^{-1}(w)) = a_0=0$.
\vspace{.05in}

\noindent \textbf{Claim.} \textit{$E_0(\psi^{-1}(w_p)) = E_1(\psi^{-1}(w_p))=0$ for all $p \in \bbF^+_{n_b}$ with $|p|\geq2$.}

\vspace{.05in}

Indeed, if $p=p_1p_2\dots p_l$, $l \geq2$, then 
\begin{equation}
\begin{split}
\psi^{-1}(w_p)&= \prod_{i=1}^{l}\psi^{-1}(w_{p_i}) \\
&=\lim_{n\rightarrow \infty} \prod_{i=1}^{l} \left(\sum_{k=1}^n \sum_{|q| = k }\left(1 - \frac{k}{n+1}\right)v_qa_{q, p_i}\right).
\end{split}
\end{equation}
Since in the limit above we have $l\geq 2 $ and $k\geq1$, a development of the product involved will reveal only terms of the form $v_u a_u$, with $u \in \bbF^+_{n_a}$, $|u|\geq 2$ and $a_u\in \A$. Since both $E_0$ and $E_1$ are continuous, this suffices to prove the claim.

\vspace{.02in}

For the proof consider $i=1,2$. Apply $\psi^{-1}$ to (\ref{eq;psi}) and use the Claim to obtain
\begin{equation*}
\begin{split}
E_i(v)&=E_i(\psi^{-1}(\psi(v)))\\
&= E_i((b_0)+E_i(\psi^{-1}(w)) [b_{ij}]+\lim_{n\rightarrow \infty} \sum_{k=2}^n \sum_{|p| = k }\left(1 - \frac{k}{n+1}\right) E_i(\psi^{-1}(w_p)) b_p \\
&= E_i(b_0) +E_i(\psi^{-1}(w))[b_{ij}] \\
&=E_i(b_0) +E_i\Big(v[a_{ij}]+\lim_{n\rightarrow \infty} \sum_{k=2}^n \sum_{|q| = k }\left(1 - \frac{k}{n+1}\right)v_qa_q \Big) [b_{ij}] \\
& = E_i(b_0) +E_i(v) [a_{ij}][b_{ij}], 
\end{split}
\end{equation*}
For $i=0$ we obtain $0=b_0=E_0(\psi(v))$. For $i=1$ we obtain $v =v [a_{ij}][b_{ij}]$ and so 
\[
I_{n_a} =v^*v=v^*v [a_{ij}][b_{ij}]= [a_{ij}][b_{ij}].
\]
By reversing the roles of $\psi$ and $\psi^{-1}$ and using what has been proven so far, we obtain $[b_{ij}].[a_{ij}] =I_{n_b}$. This completes the proof.
\end{proof}

\begin{corollary} \label{cor;secmain}
Let $\psi : \T^+(\A, \alpha) \rightarrow \T^+(\A, \beta)$ be a completely isometric isomorphism and assume that $\psi|_{\mathcal A} = \id$. Let $v$ be the generating row isometry for $\T^+(\A, \alpha) $ and assume that $E_0(\psi(v))=0$. Then there exists a unitary matrix $u \in M_{n_b, n_a}(\A)$ intertwining the endomorphisms $\alpha_1, \alpha_2, \dots , \alpha_{n_a}$ and $\beta_1, \beta_2, \dots, \beta_{n_b}$.
\end{corollary}

\begin{proof}
According to (\ref{eq;bij}), the matrix $b:=[b_{i j}] \in M_{n_b, n_a}(\A)$ associated with $\psi$ intertwines the endomorphisms $\alpha_1, \alpha_2, \dots , \alpha_{n_a}$ and $\beta_1, \beta_2, \dots, \beta_{n_b}$,  i.e., 
\begin{equation} \label{eq;intert}
\diag(\beta_1(a), \dots, \beta_{n_b}(a))b= b\diag(\alpha_1(a),\dots,\alpha_{n_a}(a)),
\end{equation}
for all $a \in \A$. Hence $b^*b$ (and therefore $|b|$) commutes with $\diag(\alpha_1(a),\dots,\alpha_{n_a}(a))$, for all $a\in \A$. By Lemma~\ref{lem;inv}, $b$ is invertible and so it admits a polar decomposition $b = u|b|$, with $u=M_{n_b, n_a}(\A)$ a unitary matrix. Furthermore, (\ref{eq;intert}) implies that
\begin{align*}
\diag(\beta_1(a), \dots, \beta_{n_b}(a)) u|b| &= u|b| \diag(\alpha_1(a),\dots,\alpha_{n_a}(a))\\
								&= u\diag(\alpha_1(a),\dots,\alpha_{n_a}(a)) |b|,
\end{align*}
for all $a \in \A$. Since $|b|$ is invertible, $\diag(\beta_1(a), \dots, \beta_{n_b}(a)) u= u \diag(\alpha_1(a),\dots,\alpha_{n_a}(a))$ and the conclusion follows.
\end{proof}

Motivated by the statement of the previous result, we introduce the following

\begin{definition} \label{maindefn}
Two multivariable dynamical systems $(\A, \alpha)$ and $(\A, \beta)$ are said to be {\em unitarily equivalent} if there exists a unitary matrix with entries in $\A$ intertwining the two systems. Two multivariable dynamical systems $(\A, \alpha)$ and $(\B, \beta)$ are said to be are {\em unitarily equivalent after a conjugation} if there exists a $*$-isomorphism $\gamma: \A\rightarrow \B$ so that the systems $(\A , \alpha)$ and $(\A , \gamma^{-1}\circ \beta \circ \gamma)$ are unitarily equivalent.
\end{definition}

Note that the above definition does not require that the multivariable systems  $(\A, \alpha)$ and $(\A, \beta)$ should have the same number of maps, i.e., $n_{\alpha}=n_{\beta}$, where $\alpha = (\alpha_1, \alpha_2, \dots , \alpha_{n_{\alpha}})$ and $\beta=(\beta_1, \beta_2, \dots, \beta_{n_{\beta}})$. On the contrary, it is possible  for two dynamical systems with a different number of maps to be unitarily  equivalent; see \cite[Example 5.1]{KakK}. Nevertheless in that  case at least one of the systems  will fail to be automorphic, as \cite[Theorem 4.4]{KakK} clearly indicates.

Recall that two (single variable) dynamical systems are said to be outer conjugate if there is a $*$-isomorphism $\gamma : \mathcal A \rightarrow \mathcal B$ and a unitary $u\in \mathcal A$ such that
\[
\alpha(c) = u(\gamma^{-1}\circ\beta\circ\gamma(c))u^*
\]
Therefore in that case the concept of unitary equivalence after a conjugation coincides with that of outer conjugacy. Davidson and Kakariadis \cite{DavKak} established that outer conjugacy of the systems implies that the associated tensor algebras are completely isometrically isomorphic. They showed that the converse is true in several broad cases (injective, surjective, etc.) and specifically when $\|E(\psi(v))\| < \frac{2}{3}\sqrt 3 - 1 \approx 0.1547$ \cite[Remark 3.6]{DavKak}. For multivariable dynamical systems consisting of automorphisms Kakariadis and Katsoulis have shown \cite[Theorem 4.5]{KakK} that the isomorphism of the tensor algebras is equivalent to unitary equivalence after a conjugation for the associated dynamical systems. A similar result was shown for arbitrary dynamical systems provided that the pertinent $\ca$-algebras are stably finite \cite[Theorem 5.2]{KakK}. Our next result removes all these conditions and establishes a complete result for the unital case.

\begin{theorem} \label{T;main}
If $\psi :\T^+(\A, \alpha) \rightarrow\T^+(\B , \beta)$ is a completely isometric isomorphism then $(\mathcal A,\alpha)$ and $(\mathcal B,\beta)$ are unitarily equivalent after a conjugation.\end{theorem}

\begin{proof}
Without loss of generality we may assume that $\A = \B$ and $\psi|_{\mathcal A} = \id$. Indeed,
it is well-known that the restriction of the isomorphism $\psi$ on $\A \subseteq \T^+(\A, \alpha)$ induces a $*$-isomorphism $\gamma : \mathcal A \rightarrow \mathcal B$ of the diagonals of the two tensor algebras. (See for instance \cite[Proposition 3.1]{DavKat1.5}.) The dynamical systems $(\B , \beta)$ and $(\A , \gamma^{-1}\circ \beta \circ \gamma)$ are conjugate via $\gamma$ and so there exists a completely isometric isomorphism $\phi: \T^+(\B , \beta) \rightarrow \T^+(\A , \gamma^{-1}\circ \beta \circ \gamma)$ so that $\phi|_{\B} = \gamma^{-1}$. Hence $\phi\circ\psi$ establishes a completely isometric isomorphism between $\T^+(\A, \alpha)$ and $\T^+(\A , \gamma^{-1}\circ \beta \circ \gamma)$, whose restriction on $\A$ is the identity map. Therefore if $\gamma$ is not the identity map to begin with, then replace $\psi$ with $\phi\circ \psi$ and establish the conclusion for the dynamical systems $(\A , \alpha)$ and $(\A , \gamma^{-1}\circ \beta \circ \gamma)$.

For the proof, if $v = [v_1\ v_2\ \dots\ v_n]$ is the generating row isometry in $\T^+(\A, \alpha) $, then Proposition~\ref{prop;strict} gives that $b_0 = E_0(\psi(v))$ is a strict row contraction. Combined with (\ref{eq;mainrel}), this implies that $b_0$ satisfies the conditions of Theorem \ref{thm:mobius} and so there exists a completely isometric automorphism $\rho$ of  $\T^+(\A, \alpha)$ such that 
\[
\rho(v) = (I - b_0b_0^*)^{1/2}(I-vb_0^*)^{-1}(b_0-v)(I_n-b_0^*b_0)^{-1/2}.
\]
 Since $E_0$ is multiplicative, we obtain that
\begin{align*}
E_0(\psi\circ \rho(v) )
&= E_0\circ\psi\left( D_{b_0^*}(I-vb_0^*)^{-1}(b_0-v)D_{b_0}\right)
\\ &=D_{b_0^*}E_0((I-\psi(vb_0^*))^{-1}\left(b_0-E_0(\psi(v))\right)D_{b_0}
\\ &=D_{b_0^*}E_0((I-\psi(vb_0^*))^{-1}\left(b_0-b_0\right)D_{b_0}
\\ & = 0.
\end{align*}
Therefore, $\psi\circ\rho$ satisfies the requirements of Corollary~\ref{cor;secmain} and the conclusion follows.
\end{proof}

Of course, the converse of Theorem~\ref{T;main} is also true. If two multivariable systems $(\mathcal A,\alpha)$ and $(\mathcal B,\beta)$ are unitarily equivalent after a conjugation, then the associated $\ca$-correspondences are unitarily equivalent and so the tensor algebras $\T^+(\A, \alpha)$ and $\T^+(\B , \beta)$ are completely isometrically isomorphic. (See \cite[Theorem 4.5]{KakK12} and the discussion preceding it.) Hence Theorem~\ref{T;main} provides a complete classification of tensor algebras up to complete isomorphism. For semicrossed products we can say something more.

\begin{corollary} \label{cor;main}
Let $(\A, \alpha)$ and $(\B, \beta)$ be two unital $\ca$-dynamical systems. Then $\A\rtimes_{\alpha} \bbZ^+$ and $\B\rtimes_{\beta} \bbZ^+$ are isometrically isomorphic if and only $(\A, \alpha)$ and $(\B, \beta)$ are outer conjugate.
\end{corollary}

\begin{proof}
The result follows from Theorem~\ref{T;main} and the fact that contractive representations of semicrossed products associated with unital endomorphisms are always completely contractive \cite[Corollary 3.14]{MS1}.
\end{proof}

\section{concluding remarks and open problems}

(i) As we mentioned in the introduction, Davidson and Kakariadis \cite[Theorem 1.1]{DavKak} provide six different properties for a dynamical system with each one of them guaranteeing that the isomorphism problem has a positive resolution. It is instructive to observe that there are dynamical systems that do not satisfy any of their properties. Therefore our Corollary~\ref{cor;main} verifies the conjecture of Davidson and Kakariadis by going beyond the realm of  \cite[Theorem 1.1]{DavKak}.

Each one of the following conditions on a $\ca$-dynamical system $(\A, \alpha)$ is sufficient for \cite[Theorem 1.1]{DavKak} to apply
\begin{itemize}
\item[(1)] $\A$ has trivial center.
\item[(2)] $\A$ is abelian.
\item[(3)] $\A$ is finite, i.e., no proper isometries.
\item[(4)] $\alpha(\A)'$ is finite. 
\item[(5)] $\alpha(R_{\alpha})=R_{\alpha}$, where $R_{\alpha}=\overline{\cup_{k \geq 1} \ker(\alpha^k )}$.
\item[(6)] $\alpha(R_{\alpha}^{\perp}) \subseteq R_{\alpha}^{\perp}$.
\end{itemize}
We claim that for each $i=1,2,\dots, 6$, there exists a dynamical system $(\A_i, \alpha_i)$ that fails the corresponding condition $(i)$ from the above list. It is easy to see then that the dynamical system $(\oplus_{i=1}^6 \A_i, \oplus_{i=1}^6 \alpha_i)$ will fail all six conditions.

The existence of dynamical systems that do not satisfy anyone of the conditions (1), (2) or (3) is a trivial matter. For condition $(4)$, let $\H$ be a separable Hilbert space and let $B(\H), K(\H)$ denote the bounded and compact operators respectively acting on $\H$.
Let 
\[
\A_4 = B(\H) \otimes K(\H) + \bbC(I\otimes I)
\]
acting on $\H \otimes \H$ and let $\alpha_4$ be the unital endomorphism of $\A$ defined as 
\[
\alpha_4 (S +\lambda I\otimes I)= \lambda I\otimes I, \quad S \in  B(\H) \otimes K(\H), \,\, \lambda \in \bbC.
\]
The dynamical system $(\A_4, \alpha_4)$ fails properrty (4) from the above list.

For condition (6), let $\H$ and $K(\H)$ be as in the previous paragraph and consider
\[
\A_6=( K(\H)+\bbC I )\oplus \bbC I
\]
acting on $\H\oplus \H$ and let $\alpha_6$ be the unital endomorphism of $\A_6$ defined as 
\[
\alpha_6(K + \lambda I, \mu I) = (\mu I , \lambda I), \quad K \in K(\H), \, \, \lambda, \mu \in \bbC.
\]
It is clear that $\ker (\alpha_6^k) = K(\H)\oplus 0$, for all $k\geq 1$, and so $R_{\alpha} = K(\H)\oplus 0$. Therefore $R_{\alpha_6}^{\perp} = 0\oplus \bbC I$ and so 
\[
\alpha_6 (R_{\alpha_6}^{\perp} )= \bbC I \oplus 0 \not\subseteq 0\oplus \bbC I = R_{\alpha_6}^{\perp},
\]
i.e., $(\A_6, \alpha_6)$ does not satisfy (6). Note also that 
\[
\alpha_6(R_{\alpha_6})= 0 \neq R_{\alpha_6}
\]
and so $(\A_6, \alpha_6)$ does not satisfy (5) as well.

\vspace{.2in}

(ii) Could one obtain Corollary~\ref{cor;main} by just using the earlier techniques of Davidson and Kakariadis \cite{DavKak}? Notice that their \cite[Remark 3.6]{DavKak} would imply Corollary~\ref{cor;main}, provided that one could establish the existence of an isomorphism $\psi$ that satisfies the technical requirement $\|E(\psi(v))\| < \frac{2}{3}\sqrt 3 - 1 \approx 0.1547$. However the existence of such an isomorphism was not established in \cite{DavKak} and it does not seem likely that the techniques of  \cite{DavKak} alone could do that. Our Theorem~\ref{thm:mobius} and Proposition~\ref{prop;strict} show now that such an isomorphism $\psi$ exists and satisfies the much stronger condition $\|E(\psi(v))\| =0$. (This actually provides a slightly different proof of Corollary~\ref{cor;main}.) Our approach in Theorem~\ref{thm:mobius} is much different from that of \cite{DavKak} and it is inspired by the works of Muhly and Solel \cite{MS2} and Davidson, Ramsey and Shalit \cite{DRS}, which were actually available during the writing of \cite{DavKak}. 

Note that a multivariable analogue of \cite[Remark 3.6]{DavKak} does not exist in the literature, thus necessitating the approach that we followed in the proof of Theorem~\ref{T;main}.

\vspace{.2in}

(iii) Even though we provide a complete invariant for completely isometric isomorphisms, Theorem~\ref{T;main} is not the end of the story for the isomorphism problem for tensor algebras of multivariable dynamical systems. There are still questions that need to be addressed and the following two seem to be the most important.

\begin{question} Does Theorem~\ref{T;main} hold for algebraic isomorphisms? Is unitary equivalence after a conjugation a complete invariant for algebraic isomorphisms between tensor algebras of multivariable dynamical systems?
\end{question}

Indeed a great deal of work on isomorphisms between tensor algebras of abelian multivariable systems addresses algebraic isomorphisms. In general, it is possible for non-selfadjoint operator algebras, even tensor  algebras, to be algebraically or bicontinuously isomorphic without being isometrically isomorphic. So the above question essentially asks whether the various concepts of isomorphism coincide for tensor algebras of multivariable dynamical systems. Remarkably to this date there has been no work addressing algebraic isomorphisms between tensor algebras of non-abelian multivariable systems. One reason for that is perhaps the fact that algebraic isomorphisms between non-selfadjoint operator algebras do not preserve the diagonal and so the techniques developed for isometric isomorphisms are not applicable here. Hopefully the progress achieved in this paper will finally facilitate the study of algebraic isomorphisms beyond the abelian case.

\begin{question}
Is piecewise conjugacy a complete invariant for completely isometric isomorphisms between tensor algebras of multivariable dynamical systems over abelian $\ca$-algebras.
\end{question}

This important question goes back to the memoirs of Davidson and Katsoulis~\cite[Conjecture 3.26]{DavKat2} and the second half of \cite{DavKat2} was actually occupied with partial solutions to this question. In light of Theorem~\ref{T;main}, this question essentially asks whether piecewise conjugacy and unitary equivalence after a conjugation coincide as invariants for multivariable dynamical systems over abelian $\ca$-algebras; in this form the question was raised in \cite[Question 2]{KakK}.  As we mentioned in the introduction of this paper, this question does not make sense in the generality of multivariable dynamical systems over-non abelian $\ca$-algebras. Hence the invariant we offer here is the only one available beyond the abelian case.


\vspace{0.1in}

{\noindent}{\it Acknowledgement.} The authors thank the reviewer for their helpful comments.

\vspace{0.05in}

{\noindent}{\it Conflict of interest statement.} 
There are no conflicts of interest.

\vspace{0.05in}

{\noindent}{\it Financial Support.} The first author was partially supported by NSF Grant 2054781 and the second author was supported by NSERC Discovery Grant 2019-05430. 


\end{document}